\documentclass[12pt]{article}

\usepackage{amsmath,amsthm,amsfonts,latexsym,amscd,amssymb}

\def\qed{{\hbadness=10000\hfill\ \vbox{\hrule height.09ex
     \hbox{\vrule width.09ex height1.55ex depth.2ex \kern1.8ex
     \vrule width.09ex height1.55ex depth.2ex}\hrule height.09ex}\break
     \bigskip}}
\newtheorem{theorem}{Theorem}[section]
\newtheorem{lemma}[theorem]{Lemma}

\newtheorem{proposition}[theorem]{Proposition}

\newtheorem{remark}[theorem]{Remark}

\begin{document}
\title{Some Characterizations of a Normal Subgroup of a Group}

\author{Vipul Kakkar\footnote{The first author is supported by CSIR, Government of India.}
~ and R.P. Shukla\\
Department of Mathematics, University of Allahabad \\
Allahabad (India) 211 002\\
Email: vplkakkar@gmail.com; shuklarp@gmail.com}

\date{}
\maketitle

\begin{abstract}
Let $G$ be a group and $H$ be a subgroup of $G$ which is either finite or of finite index in $G$. In this note, we give some characterizations for normality of $H$ in $G$.  As a consequence we get a very short and elementary proof of the Main Theorem of \cite{rjpf}, which avoids the use of the classification of finite simple groups.
\end{abstract}
\noindent \textbf{\textit{Key words:}} Right loop, Normalized Right Transversal, Right Inverse Property.
\section{Introduction}
Let $G$ be a group and $H$ be a subgroup of $G$. A \textit{normalized right transversal (NRT)} $S$ of $H$ in $G$ is a subset of $G$ obtained by choosing one and only one element from each right coset of $H$ in $G$ and $1 \in S$. Then $S$ has a induced binary operation $\circ$ given by $\{x \circ y\}=Hxy \cap S$, with respect to which $S$ is a right loop with identity $1$, that is, a right quasigroup with both sided identity  (see \cite[Proposition 4.3.3, p.102]{smth},\cite{rltr}). Conversely, every right loop can be embedded as an NRT in a group with some universal property (see \cite[Theorem 3.4, p.76]{rltr}).

Let $\mathcal{T}(G, H)$ denote the set of all NRTs (normalized right transversals) of $H$ in $G$. We say that $S$ , $T \in \mathcal{T}(G, H)$ are isomorphic, if their induced right loop structures are isomorophic. If $H$ is normal subgroup in $G$, then each $S \in \mathcal{T}(G, H)$ is isomorphic to the quotient group $G/H$. The converse of this statement was proved in \cite[Main Theorem, p.643]{rjpf} for finite groups:

\begin{theorem}[Main Therem \cite{rjpf}]\label{th1}
Let $G$ be a finite group and $H$ a subgroup of $G$. If all NRTs of $H$ in $G$  are isomorphic, then $H$ is normal in $G$.
\end{theorem}
  
 The proof of the Main Theorem in \cite{rjpf} used the classification of finite simple groups (the knowledge of order of automorphism groups of finite non-abelian simple groups). In this note, we obtain an elementary short proof of Theorem \ref{th1}, which avoids the use of the classification of finite simple groups.

Let $S \in \mathcal{T}(G, H)$. For $x \in S$, we denote the map $y\mapsto y \circ x$ $(y \in S)$ by $R_x$, where $\circ$ is the binary operation on $S$ defined in the first paragraph of Section 1. We say that (1) $S$ has \textit{right inverse property(RIP)}, if there is a map $r:S\rightarrow S$ such that $R_x^{-1}=R_{r(x)}$, for all $x \in S$, (2) $S$ is \textit{right conjugacy closed (RCC)}, if for each pair $(x,y) \in S \times S$ there exists $z \in S$ such that $R_xR_yR_x^{-1}=R_z$, (3) $S$ is \textit{$A_r$-transversal} if $H \subseteq N_G(S)$, where $N_G(S)$ denotes the normalizer of $S$ in $G$. Now, we state the main result of this note: 

\begin{theorem}\label{mth} Let $H$ be subgroup of $G$ such that either the order $\left|H\right|$ of $H$ or the index $[G:H]$ is finite. Then following are equivalent:
\begin{enumerate}
\item $H$is a normal subgroup of $G$.
\item  All $S \in \mathcal{T}(G, H)$ are both sided transversals.
\item  All $S \in \mathcal{T}(G, H)$ are isomorphic.
\item  All $S \in \mathcal{T}(G, H)$ have RIP.
\item  All $S \in \mathcal{T}(G, H)$ are RCC.
\end{enumerate}
\end{theorem}

\section{Proof of the Theorem \ref{mth}}
Let $G$ be a group and $H$ a subgroup of $G$. It is shown in \cite{hwc} and \cite{gam}
that if $H$ is finite subgroup of $G$ then there exists a common set of representatives for the left and right cosets of $H$ in $G$. Let us call such a transversal as \textit{both sided transversal}. In \cite[Theorem 3, p. 12]{hz}, it is observed that if the index $[G:H]$ of $H$ in $G$ is finite, then both sided transversal exists. O. Ore has generalized these results in \cite{ore}.

Let $S \in \mathcal{T}(G, H)$ and $\circ$ be the binary operation on $S$ defined in the first paragraph of Section 1. Let $x, y\in S$ and $h\in H$. Then $x.y = f(x, y)(x\circ y)$ for some $f(x, y)\in H$ and $x\circ y\in S$. Also $x.h=\sigma_x (h) x\theta h $ for some $\sigma_x (h)\in H $ and $x\theta h \in S$. This gives us a map $f:S\times S\rightarrow H$ and a map $\sigma :S\rightarrow H^H$ defined by $f((x,y))=f(x, y)$ and $\sigma (x)(h)=\sigma_x (h)$. Also $\theta$ is a right action of $H$ on $S$. The quadruple $(S,H,\sigma,f)$ is a $c$-groupoid (see \cite[Definition 2.1, p.71]{rltr}). Infact, every $c$-groupoid comes in this way (see \cite[Theorem 2.2, p.72]{rltr}). The same is observed in \cite{km} but with different notations (see \cite[Section 3, p. 289]{km}). We need following result of \cite{km} to prove Theorem \ref{mth}:

\begin{proposition}[\cite{km}, Proposition 3.5, p. 292]\label{p1}Let $S \in \mathcal{T}(G, H)$ and $(S,H,\sigma,f)$ be the associated $c$-groupoid. Then following are equivalent:
\begin{enumerate}
\item $\sigma_x:S\rightarrow S$ is surjective, for all $x \in S$.
\item The equation $x \circ X=1$, where $X$ is unknown, has a solution, for all $x \in S$.
\item $S$ is a both sided transversal. 
\end{enumerate}
\end{proposition}
The equivalence of (2) and (3) has also been proved in \cite[Lemma 7*, p.30]{ek}

\begin{lemma}\label{l1}
Let $S,T \in \mathcal{T}(G, H)$ be isomorphic and $S$ be a both sided transversal. Then $T$ is also both sided transversal. 
\end{lemma}
\begin{proof} Let $\circ$ and $\circ^{\prime}$ be the induced binary operations on $S$ and $T$ respectively. Fix an isomorphism $p:S \rightarrow T$. Let $y \in T$ and $x=p^{-1}(y)$. Since $S$ is a both sided transversal, by Proposition \ref{p1} there exists $a \in S$ such that $x \circ a=1$. Hence $y \circ^{\prime} p(a)=p(x) \circ^{\prime} p(a)=p(1)=1$. Thus by Proposition \ref{p1}, $T$ is a both sided transversal. 
\end{proof}
\begin{lemma}\label{l2}
Let $G$ be a group and $H$ be a non-normal subgroup of $G$. Then there exists $S \in \mathcal{T}(G, H)$, which is not a left transversal of $H$ in $G$. 
\end{lemma}
\begin{proof} Since $H \ntrianglelefteq G$, there exists $x \in G$ such that $xH \neq Hx$. We may assume that $xH \setminus Hx \neq \emptyset$, for $xH \subsetneq Hx$ if and only if $Hx^{-1} \subsetneq x^{-1}H$ (and so we may replace $x$ by $x^{-1}$, if necessary). Choose $y \in xH \setminus Hx$. Then $xH=yH$ but $Hx \neq Hy$. Let $S \in \mathcal{T}(G, H)$ containing $1,x,y$. Clearly $S$ is a right transversal but not left transversal of $H$ in $G$. 
\end{proof}
\begin{proposition}\label{p2} Let $G$ be a group and $H$ be a subgroup of $G$. If all $S \in \mathcal{T}(G, H)$ are $A_r$-transversals, then $H \trianglelefteq G$.
\end{proposition}
\begin{proof} Assume that all $S \in \mathcal{T}(G, H)$ are $A_r$-transversals. Let $S \in \mathcal{T}(G, H)$. Let $x \in S$ and $h\in H$. Since $S$ is an $A_r$-transversal, $h^{-1}xh \in S$. Hence $xh=h(h^{-1}xh)$ implies that $\sigma_x=I_H$, for all $x \in S$, where $I_H$ is the identity map on $H$. By Proposition \ref{p1}, all $S$ are both sided transversals. Thus by Lemma \ref{l2}, $H \trianglelefteq G$.
\end{proof}
\begin{remark}\label{r1} The converse of Proposition \ref{p2} is not true. For example, let $G=Sym(3)$, the symmetric group of degree $3$. Let $H$ and $S$ be subgroups of $G$ of order $3$ and $2$ respectively. Then $H \trianglelefteq G$, $S \in \mathcal{T}(G,H)$ and $N_G(S)=S$. Thus $S$ is not an $A_r$-transeversal.
\end{remark}
\noindent \textit{Proof of Theorem \ref{mth}}:~ The statement (1) implies each of the statements (2)-(6) (for, all $S \in \mathcal{T}(G, H)$ are isomorphic to the group $G/H$). 
\\
$2\Rightarrow 1$: Follows from Lemma \ref{l2}.
\\
$3\Rightarrow 2$: Assume that (3) holds. Let $S$ be a both sided transversal of $H$ in $G$ (See first paragraph of Section 2). By Lemma \ref{l1}, all $S \in \mathcal{T}(G, H)$ are both sided transversals. 
\\
$4\Rightarrow 2$: Assume that (4) holds. Let $S \in \mathcal{T}(G, H)$. Since $S$ has RIP, there exists a map $r:S\rightarrow S$ such that $R_x^{-1}=R_{r(x)}$ for all $x \in S$. Fix $x \in S$. Then $R_xR_{r(x)}(x)=x$, that is, $(x \circ r(x))\circ x=x=1 \circ x$, where $\circ$ is the binary operation on $S$ defined in the first paragraph of Section 1. By right cancellation in $S$, $x \circ r(x)=1$. Hence by Proposition \ref{p1}, $S$ is a both sided transversal. Thus (2) holds. 
\\
$5\Rightarrow 2$: Assume that (5) holds. Let $S \in \mathcal{T}(G, H)$ and $x^{\prime}$ be the left inverse of $x$ in $S$. Since $S$ is RCC, there exists $z\in S$ such that $R_xR_{x^{\prime}}R_x^{-1}=R_z$. Hence $R_xR_{x^{\prime}}R_x^{-1}(x)=R_z(x)$, i.e. $x \circ z=x^{\prime} \circ x=1$, where $\circ$ is the binary operation on $S$ defined in the first paragraph of Section 1. Hence by Proposition \ref{p1}, $S$ is a both sided transversal. Thus (2) holds. 
\qed
\begin{center}\textbf{Acknowledgement}\end{center}
Authors are grateful to Prof. Yoav Segev for kindly pointing out an error in the earliar version of the manuscript. This resulted Remark \ref{r1}.



\begin{thebibliography}{99}

\bibitem{hwc}H. W. Chapman
{\em A note on the elementary theory of groups of finite order},
Messenger of Math., 42 , 132-134 (1913).


\bibitem{km} J. Klim and S. Majid,
{\em Bicrossproduct Hopf quasigroups},
Comm. Math. U. Carolinas, 51 ,287-304 (2010).

\bibitem{ek} E.A. Kuznetsov,
{\em Transversals in groups. 1: Elementary properties},
Quasigroups Related Systems 1 , 22-42 (1994).

\bibitem{rltr}R. Lal,
{\em Transversals in Groups},
J. Algebra 181 , 70-81 (1996).


\bibitem{rjpf} R. Lal and R. P. Shukla,
{\em Perfectly stable  subgroups of finite groups},
Comm. Algebra 24(2) , 643-657 (1996).

\bibitem{gam} G. A. Miller,
{\em On a method due to Galois},
Quart, J. Math. Oxford Ser. 41 , 382-384 (1910).

\bibitem{ore} O. Ore,
{\em On coset representatives in groups},
Proc. Amer. Math. Soc., 9 , 665-670 (1958).

\bibitem{smth} J. D. H. Smith and Anna B. Romanowska,
{\em Post-Modern Algebra},
John Wiley \& Sons, Inc., 1999.

\bibitem{hz}H. Zassenhaus,
{\em The Theory of Groups},
Chelsea, New York, 1949.
\end{thebibliography}
\end{document}